\documentclass{amsart}
\usepackage{amsthm}
\usepackage{amssymb}
\usepackage{amsmath}
\usepackage{amsmath,amscd}
\usepackage{color}
\usepackage{hyperref}
\usepackage{tikz}
\usepackage{mathtools}
\usepackage{tikz-cd}

\newtheorem{thm}{Theorem}[section]
\newtheorem{cor}[thm]{Corollary}
\newtheorem{prop}[thm]{Proposition}
\newtheorem{lem}[thm]{Lemma}

\theoremstyle{definition}

\newtheorem{exmps}[thm]{Examples}

\newtheorem{remark}[thm]{Remark}

\newcommand{\Fl}{\mathrm{Fl}}
\newcommand{\Gr}{\mathrm{Gr}}

\begin{document}

\title[Factorial Schur and Grothendieck polynomials]{Factorial Schur and Grothendieck polynomials from Bott Samelson varieties}
\author{David Oetjen}
\address{
Department of Mathematics,
Virginia Tech University,
Blacksburg, VA 24061
USA}
\email{doetjen@vt.edu}
\maketitle

\begin{abstract}
We use Bott-Samelson resolutions of Schubert varieties in Grassmannians along with equiariant localization techniques to show that the factorial Schur functions and the factorial Grothendieck polynomials represent Schubert classes in equivariant cohomology and equivariant K-theory respectively.
\end{abstract}

\tableofcontents
\section{Introduction}

It is known (e.g. \cite[Chapter 9.4]{Ful:YT}) that the Schur functions represent the Schubert classes in the cohomology ring of the Grassmannians. In equivariant cohomology, it is the factorial Schur functions which represent the equivariant Schubert classes. The first half of this paper aims to reprove this fact with a new method, using a particular type of desingularization of Schubert varieties along with equivariant localization. The second half will follow the same path in equivariant K-theory, showing that the factorial Grothendieck polynomials represent the classes of the equivariant structure sheaves of the Schubert varieties.

\subsection{Statement of results}

Let $\lambda=(\lambda_1,...,\lambda_k)$ be a partition, meaning $\lambda_i\geq\lambda_j$ for $i<j$ and $\lambda_i\in\mathbb{Z}_{\geq0}$ for all $1\leq i\leq k$.
Let $N\in\mathbb{N}$ and $n=N+k$.
Then let $\Omega^\lambda\subseteq\Gr(k,n)$ denote the Schubert variety for $\lambda$, as defined in subsection 2.2.
Let $\pi:\mathbb{V}^\lambda\to\Omega^\lambda$ be the desingulariation as defined in subsection 2.3.
By construction, for a partition $\mu\subseteq\lambda$, we have that $\mathbb{V}^\lambda\subseteq\mathbb{V}^\mu$.
Furthermore, there are line bundles $\{\mathcal{L}_i\}_{1\leq i\leq k}$ on $\mathbb{V}^\lambda$ whose fibers are $S_i/S_{i-1}$, where $S_1\subseteq...\subseteq S_k$ parameterize the points in the partial flag variety $F\ell(1,...,k;n)$.
We then define, for $1\leq i\leq k$, $z_i=c_1^T(\mathcal{L}_i^\vee)$ in $\mathbb{V}^\lambda$ and $x_i$ to be the Chern roots of $\mathcal{S}^\vee$, the dual of the tautological sub-bundle in $\Gr(k,n)$.
In both varieties, $t_i$ is used to denote the equivariant parameters for $1\leq i\leq n$, defined by $t_i=c_1^T(\mathbb{C}_i)$, where $\mathbb{C}_i$ refers to the vector bundle whose fiber is the coordinate axis $\langle e_i\rangle$.
When used in the context of K-theory, $z_i=1-[\mathcal{L}_i]$, $x_i$ are the "K-theoretic Chern roots" which satisfy $\lambda_{-1}(\mathcal{S})=\prod_{i=1}^kx_i$, where \[\lambda_{-1}(V)=\sum_{i\geq0}(-1)^i[\Lambda^iV],\] and $T_i=1-[\mathbb{C}_i^\vee]$.

Factorial Schur functions were first introduced by Biedenharn and Louck in \cite{BL:factSchur}. Macdonald generalized those functions into the form used in this paper, in \cite[6th Variation]{MacD:Schur}:
\[s_\lambda(x,t)=\frac{\text{det}((x_i|t)^{\lambda_j+k-j})}{\prod_{1\leq i<j\leq k}(x_i-x_j)},\]
where $(x_i|t)^r=\prod_{j=1}^r(x_i+t_j)$.
Those factorial Schur functions have been studied extensively (e.g. \cite{GG:factSchur}, \cite{Molev:factSchur}, \cite{MS:factSchur}, and \cite[Chapter I.3]{Mac:SF}).
It is known that the double Schubert polynomials introduced by Lascoux and Sch\"{u}tzenberger \cite{LS:Schubert} represent the equivariant Schubert classes in flag varieties (e.g. \cite{Billey:Kostant}, \cite{Arabia:Kac-Moody}), and that the double Schubert polynomials coincide with the factorial Schur functions for Grassmannian permutations. These together imply that the factorial Schur functions represent the equivariant Schubert classes. See e.g. \cite[Section 5]{Mihalcea.factSchur}.

Factorial Grothendieck polynomials were first introduced by McNamara using set-valued tableaux in \cite{McN:factGrothendieck}, then Ikeda and Naruse formulated them as a determinant in \cite{IN:factGrothendieck}: \[G_\lambda(x,y)=\frac{\text{det}((x_i|T)^{\lambda_j+k-j}(1-x_i)^{j-1})}{\prod_{1\leq i<j\leq k}(x_i-x_j)},\]
where $(x_i|T)^r=\prod_{j=1}^r(x_i+T_j-x_iT_j)$.
The factorial Grothendieck polynomials coincide with the double Grothendieck polynomials for Grassmannian permutations (e.g. \cite[Theorem 8.7]{McN:factGrothendieck}), which also represent the Schubert classes in equivariant K-theory of flag varieties \cite{LS:Grothendieck}. From this one deduces that the factorial Grothendieck polynomials represent the Schubert classes in equivariant K-theory of Grassmannians.

Our aim is to reprove these results using the desingularization $\pi:\mathbb{V}^\lambda\to\Omega^\lambda$. Since $\pi:\mathbb{V}^\lambda\to\Omega^\lambda$ is a desingularization and $\Omega^\lambda$ has only rational singularities, we have that in equivariant cohomology, $\pi_*([\mathbb{V}^\lambda])=[\Omega^\lambda]$, and in equivariant K-theory, $\pi_*([\mathcal{O}_{\mathbb{V}^\lambda}])=[\mathcal{O}_{\Omega^\lambda}]$. This allows us to compute the Schubert class by computing the pushforward of the class of $\mathbb{V}^\lambda$. To do this, the first step is to express the classes, $[\mathbb{V}^\lambda]\in H^*_T(\mathbb{V}^\emptyset)$ in cohomology and $[\mathcal{O}_{\mathbb{V}^\lambda}]\in K_T(\mathbb{V}^\emptyset)$ in K-theory, in terms of the $z$ and $t$ variables ($T$ variables in K-theory), which can be done by taking advantage of the structure of $\mathbb{V}^\emptyset$ as a tower of projective bundles:
\begin{lem}
For a partition $\lambda=(\lambda_1,...,\lambda_k)$, \[[\mathbb{V}^\lambda]=\prod_{i=1}^k\prod_{j=k+1-i}^{k-i+\lambda_i}(z_i+t_j)\ \text{and}\ [\mathcal{O}_{\mathbb{V}^\lambda}]=\prod_{i=1}^k\prod_{j=k+1-i}^{k-i+\lambda_i}(z_i+T_j-x_iT_j).\]
\end{lem}
These results are Lemma \ref{dubprod} in cohomology and Lemma \ref{Kpf} in K-theory.
For convenience we define \[p_\lambda(z,t)=\prod_{i=1}^k\prod_{j=k+1-i}^{k-i+\lambda_i}(z_i+t_j)\ \text{and}\ P_\lambda(z,T)=\prod_{i=1}^k\prod_{j=k+1-i}^{k-i+\lambda_i}(z_i+T_j-z_iT_j)\] for any composition (finite sequence of nonnegative integers) $\lambda$.
Then the previous lemma can be restated as: when $\lambda$ is a partition, $[\mathbb{V}^\lambda]=p_\lambda(z,t)$ in equivariant cohomology and $[\mathcal{O}_{\mathbb{V}^\lambda}]=P_\lambda(z,T)$ in equivariant K-theory.
Then by calculating the pushforward of these classes, we obtain the main result:

\begin{thm}
For a composition $\lambda$, we have that \[\pi_*(p_\lambda(z,t))=\frac{\text{det}((x_i|t)^{\lambda_j+k-j})}{\prod_{1\leq i<j\leq k}(x_i-x_j)}\] and that \[\pi_*(P_\lambda(z,T))=\frac{\text{det}((x_i|T)^{\lambda_j+k-j}(1-x_i)^{j-1})}{\prod_{1\leq i<j\leq k}(x_i-x_j)}.\]
\end{thm}

This is proven by using equivariant localization to calculate the localizations of the pushforward and then using the injectivity of the localization map to reconstruct the class of the pushforward. That process is done in Lemmas \ref{pfloc} and \ref{pfpoly} in equivariant cohomology and Lemmas \ref{Kpfloc} and \ref{Kpfpoly} in equivariant K-theory. Then applying these results to the particular polynomials $p_\lambda(z,t)$ and $P_\lambda(z,T)$ gives the result.
The pushforward for arbitrary classes may also be written as Jacobi-like symmetrizing operators; see the Appendix for details.

If $\lambda$ is a partition, the right sides of the equations in the Theorem are exactly the factorial Schur and factorial Grothendieck polynomials. However, if $\lambda$ is not a partition, there is a straightening rule to express the pushforward of nonpartitions in terms of the pushforward of partitions. This formula may be of interest in its own right.
\begin{thm}
For a composition $\lambda=(\lambda_1,...,\lambda_k)$ and any $1\leq i\leq k-1$, we have \[\pi_*(p_\lambda(z,t))=-\pi_*(p_\mu(z,t))\] in $H^T(\Gr(k,n))$, where $\mu=(\lambda_1,...,\lambda_{i+1}-1,\lambda_i+1,...,\lambda_k)$ and \[\pi_*(P_\lambda(z,T))=\sum_{j=\lambda_i+1}^{\lambda_{i+1}}\frac{1-T_{j+k-i}}{1-T_{\lambda_{i+1}+k-i}}\pi_*(P_{\mu^{(j)}}(z,T))-\sum_{j=\lambda_i+1}^{\lambda_{i+1}-1}\frac{1-T_{j+k-i}}{1-T_{\lambda_{i+1}+k-i}}\pi_*(P_{\nu^{(j)}}(z,T))\] in $K_T(\Gr(k,n))$, where $\mu^{(j)}=(\lambda_1,...,\lambda_{i+1},j,...,\lambda_k)$ and $\nu^{(j)}=(\lambda_1,...,\lambda_{i+1}-1,j,...,\lambda_k)$.
\end{thm}

This can be proven purely combinatorially from the determinantal formulae.
A nonequivariant version of the K-theoretic statement was proven by Buch \cite[Lemma 3.2]{Buch:sruleK}, and this equivariant form was proven by Gourbounov and Korff \cite[Corollary 2.4]{GK:Grothendieck}.
The straightening rule for Schur functions is well-known (e.g. \cite[page 40]{Mac:SF}), and the same reasoning naturally extends to factorial Schur functions.
This cohomological straightening rule also appears nonequivariantly in the calculation of the Chern-Schwartz-MacPherson classes using the same Bott-Samelson resolutions in \cite[Lemma 2.6]{aluffi.mihalcea:chern}.
In the related calculation of the Segre-MacPherson classes of matrix Schubert cells, the same straightening rule also appears in \cite[Subsection 6.1]{FR:degenloci}.
In the future we plan to use the tools established in this paper to investigate the Segre-MacPherson classes in equivariant cohomology as well as the motivic Chern classes in equivariant K-theory.

$Acknowledgements.$ I would like to thank Anders Buch and Mark Shimozono for their helpful comments. I would also like to thank my advisor, Leonardo Mihalcea, for his guidance.

\section{Preliminaries}

This section is to review some known results and establish notation for equivariant cohomology, Schubert cells and varieties, and the Bott-Samelson desingularizations of Schubert varieties.

\subsection{Equivariant cohomology}
Let X be a complex algebraic variety with a left $G$-action for $G$ an algebraic group, the $G$-equivariant cohomology of X is given by \[ H^*_G(X)=H^*(\mathbb{E}G\times^GX),\]
where $\mathbb{E}G$ is a contractible space with a free right $G$-action, and $\mathbb{E}G\times^GX=\mathbb{E}G\times X/(e\cdot g,x)\sim(e,g\cdot x)$.
For more information, refer to \cite{Anderson.eqcohom} and the references therein, but for our purposes here I will just list the needed facts.

In our case the group is $T\cong(\mathbb{C}^*)^n$, the group of invertible diagonal $n\times n$ matrices acting on $\mathbb{C}^n$ in the usual way, with this action extending to Grassmannians by acting on the subspaces. The $T$-equivariant cohomology of a point is $H_T^*(pt)\cong\mathbb{Z}[t_1,...,t_n]$, where $t_1,...,t_n$ are the generators of the weight lattice of T, and $H_T^*(X)$ is an $H_T^*(pt)$-algebra for all spaces $X$ mentioned in this paper.

Given a space $X$ and a closed, irreducible subvariety $Y\subseteq X$ invariant under the $T$-action, there is an equivariant fundamental class $[Y]\in H^{2\text{codim}(Y)}_T(X)$ associated to $Y$. Also for any equivariant, proper morphism $f:X\to Y$, there is a pushforward $f_*:H^i_T(X)\to H^{i+\text{dim}(Y)-\text{dim}(X)}_T(Y)$, and for any equivariant morphism, there is a pullback $f^*:H^*_T(Y)\to H^*_T(X)$. In particular for birational morphisms, the pushforward satisfies $\pi_*([Y])=[\pi(Y)]$. Also for any vector bundle $E$ on a space $X$, there is an equivariant total Chern class $c^T(E)\in H^*_T(X)$ which satisfies $\pi^*(c^T(E))=c^T(\pi^*(E))$, where $\pi^*(E)$ is the pullback bundle.

The inclusion map from the set of fixed points $X^T$ into $X$ is equivariant, and so the map induces a pullback map on the equivariant cohomology $\iota^*:H^*_T(X)\to H^*_T(X^T)$. If there are finitely many fixed points, then \[H^*_T(X^T)\cong\oplus_{x\in X^T}H^*_T(x),\] since the points will be disconnected from each other. In smooth varieties with finitely many fixed points and finitely many one-dimensional orbits, this map is injective \cite[Theorem 1.2.2]{GKM}, and so a class can be identified uniquely by its image under this map.
For each $x\in X^T$, the inclusion is equivariant and so induces a pullback map.
The image of a class under this map is known as the localization of the class at $x$ and is denoted \[\iota^*_x(\kappa)=\kappa|_{x}.\]
For a closed subvariety $Y\subseteq X$ and fixed point $x\in X^T$, we have that $[Y]|_x=0$ whenever $x\notin Y$. Furthermore, the fundamental classes of the fixed points generate the cohomology in the localization of the ring at $\mathbb{Z}[t_1,...,t_n]$. In particular, any class can be expressed as \[\kappa=\sum_{x\in X^T}\kappa|_x\frac{[x]}{[x]|_x}.\]

\subsection{Schubert cells and varieties} The Grassmannian of k-planes in $\mathbb{C}^n$ is the set of linear subspaces of $\mathbb{C}^n$: \[\Gr(k,n)=\{S\subseteq\mathbb{C}^n: \text{dim}(S)=k\}.\]
There is a tautological sequence of vector bundles \[0\to\mathcal{S}\to\mathbb{C}^n\to\mathcal{Q}\to0,\] where the fiber of the tautological sub-bundle $\mathcal{S}$ at some point $S\in\Gr(k,n)$ is the vector space $S$, and the fiber of the tautlogical quotient bundle $\mathcal{Q}$ is the quotient $\mathbb{C}^n/S$.
Given a partition $\lambda=(\lambda_1\geq...\geq\lambda_k\geq0)$,  and $N\geq \lambda_1$, there is a Schubert cell $\Omega^{\lambda,\circ}$ of codimension $|\lambda|=\sum\limits_{i=1}^k\lambda_i$ in the Grassmannian $\Gr(k,N+k)$ defined by \[\{S\in\Gr(k,N+k):\dim(S\cap F_{N+i-\lambda_i})=i, \dim(S\cap F_{N+i-1-\lambda_{i}})=i-1,1\leq i\leq k\}\]
For some complete flag $F_1\subset...\subset F_{N+k}$, $\dim(F_i)=i$ for $1\leq i\leq N+k$. For our purposes, we use the opposite flag, $F_i=\langle e_n,...,e_{n+1-i}\rangle$.
The closure of the Schubert cell $\Omega^{\lambda,\circ}$ is the Schubert variety $\Omega^\lambda$, which is the disjoint union of Schubert cells \[\Omega^\lambda=\{S\in\Gr(k,N+k):\dim(S\cap F_{N+i-\lambda_i})\geq i,1\leq i\leq k\}=\bigcup\limits_{\beta\geq\lambda}\Omega^{\beta,\circ}.\]
The equivariant fundamental classes of the Schubert Varieties form a $H^*_T(pt)$-basis for the equivariant cohomology ring of the Grassmannian.

\subsection{Bott-Samelson varieties}

There is also a corresponding Bott-Samelson variety in the partial flag manifold $\Fl(1,...,k;N+k)$, \[\mathbb{V}^\lambda:=\{S_1\subset...\subset S_k: \dim(S_i)=i,S_i\subseteq F_{N+i-\lambda_i}, 1\leq i\leq k\}.\]
Define a map $\pi:\Fl(1,...k;N+k)\to\Gr(k,N+k)$ by $\pi(S_1\subset...\subset S_k)=S_k$.
By the conditions $S_i\subset S_k$, $\dim(S_i)=i$, and $S_i\subseteq F_{N+i-\lambda_i}$, we have $\dim(S_k\cap F_{N+i-\lambda_i})\geq i$, and so $\pi(\mathbb{V}^\lambda)\subseteq\Omega^\lambda$.
Then for any $S\in\Omega^\lambda$, define $S_i'=S\cap F_{N+i-\lambda_i}$.
By the Schubert variety conditions, $\dim(S_i')\geq i$, and so there exist subspaces $S_i\subseteq S_i'$ such that $\dim(S_i)=i$.
Then $S_1\subset...\subset S_k\in\mathbb{V}^\lambda$ and $\pi(S_1,...,S_k)=S$.
With this, $\pi(\mathbb{V}^\lambda)=\Omega^\lambda$.
In addition, for $S\in\Omega^{\lambda,\circ}$, there is a unique $S_1\subset...\subset S_k$ such that $\pi(S_1\subset...\subset S_k)=S$, specifically $S_i=S\cap F_{N+i-\lambda_i}$, since $\dim(S_i)=i$ in this case by the conditions on $\Omega^{\lambda,\circ}$.
So then $\pi(\mathbb{V}^\lambda)=\Omega^\lambda$, and $\pi$ is an isomorphism when restricted to $\pi^{-1}(\Omega^{\lambda,\circ})$.

Next we recall that $\mathbb{V}^\lambda$ can be constructed as a tower of projective bundles, so it is smooth. Furthermore $\pi$ is a birational morphism when restricted to $\mathbb{V}^\lambda$, so $\mathbb{V}^\lambda$ is a desingularization of $\Omega^\lambda$. This construction is similar to the one in \cite{aluffi.mihalcea:chern} but uses different conventions.
For this denote by $\mathcal{F}_i$ by the trivial bundle whose fiber is $F_i$ over the relevant space.
Start by defining $\mathbb{V}_1^\lambda:=\mathbb{P}(F_{N+1-\lambda_i})$ with tautological sequence \[0\to\mathcal{O}(-1)=\mathcal{L}_1\to F_{N+1-\lambda_i}\to\mathcal{Q}_1\to0.\]
Then for $2\leq i\leq k$, define the projective bundle \[p:\mathbb{V}_i^\lambda\to\mathbb{V}_{i-1}^\lambda\]
as follows: 
Take the bundles $\mathcal{L}'_j$ and $\mathcal{Q}'_j$ on $V_{i-1}^\lambda$ for $1\leq j\leq i-1$, then define \[\mathbb{V}_i^\lambda=\mathbb{P}((\mathcal{F}_{N+i-\lambda_i}/\mathcal{F}_{N+i-1-\lambda_{i-1}})\oplus\mathcal{Q}'_{i-1})\]
with tautological sequence \[0\to\mathcal{L}_i\to(\mathcal{F}_{N+i-\lambda_i}/\mathcal{F}_{N+i-1-\lambda_{i-1}})\oplus\mathcal{Q}'_{i-1}\to\mathcal{Q}_i\to0,\]
and then define \[\mathcal{Q}_j=p^*(\mathcal{Q}'_j)\ \text{and}\ \mathcal{L}_j=p^*(\mathcal{L}'_j)\] for $1\leq j\leq i-1$. Then define $\mathbb{V}^\lambda=\mathbb{V}^\lambda_k$.

With this, for a point $(S_1\subset...\subset S_k)\in\mathbb{V}^\lambda$, the fiber of the bundle $\mathcal{Q}_i$ at the point is $F_{N+i-\lambda_i}/S_i$ and the fiber of $\mathcal{L}_i$ is $S_i/S_{i-1}$ for $1\leq i\leq k$, using $S_0=\{0\}$.

\section{Pushforward formulae}

In this section we calculate the class of $\mathbb{V}^\lambda$ as a subvariety of $\mathbb{V}^\emptyset$ (Lemma \ref{dubprod}) as well as the localizations of the pushforward of any class (Lemma \ref{pfloc}). This leads to calculating the pushforward of any class (Lemma \ref{pfpoly}).

The equivariant cohomology ring of $\mathbb{V}^\emptyset$ is generated over $H^*_T(pt)$ by the first Chern classes of the dual line bundles $\mathcal{L}_i^\vee$, which we denote by $z_i=c_1^T(\mathcal{L}_i^\vee)$ for $1\leq i\leq k$. As a result, any class in $H^*_T(\mathbb{V}^\emptyset)$ can be expressed as a polynomial in $\mathbb{Z}[t_1,...,t_{N+k}][z_1,...,z_k]$. One useful thing to know for calculating pushforwards is that \[\pi_*([\mathbb{V}^\lambda])=[\Omega^\lambda],\] since $\mathbb{V}^\lambda$ maps birationally onto $\Omega^\lambda$ under $\pi$. The following lemma expresses $[\mathbb{V}^\lambda]$ as a polynomial in the $z_i$'s:
\begin{lem}\label{dubprod}
For a partition $\lambda=(\lambda_1\geq...\geq\lambda_k)$ with $N\geq \lambda_1$, \[[\mathbb{V}^\lambda]=\prod_{i=1}^k\prod_{j=k+1-i}^{k-i+\lambda_i}(z_i+t_j)\cap[\mathbb{V}^\emptyset]\]
\end{lem}
\begin{proof}
We proceed by induction on $n$ in $\mathbb{V}^\lambda_n$. The base case $n=0$ is a point, which cannot have nonempty, proper subvarieties, so this holds trivially. The induction hypothesis is that \[[\mathbb{V}^\lambda_n]=\prod_{i=1}^n\prod_{j=k+1-i}^{k-i+\lambda_i}(z_i+t_j)\in H^*_T(\mathbb{V}^\emptyset).\]
If we denote by $\mathcal{S}_i$ the vector bundle on any subvariety of $Fl(1,...,n;N+k)$ whose fiber is $S_i$, where $Fl(1,...,n;N+k)=\{(S_1\subset...\subset S_n):\text{dim}(S_i)=i\}$, then by the construction of $\mathbb{V}^\lambda$, we have that $\mathbb{V}^\lambda_{n+1}=\mathbb{P}(\mathcal{F}_{N+n+1-\lambda_{n+1}}/\mathcal{S}_n)$ as a projective bundle over $\mathbb{V}^\lambda_n$.
Then for $\lambda'=(\lambda_1,...,\lambda_n,0,...,0)$, we have that as a projective bundle over $\mathbb{V}^\lambda_n$, $\mathbb{V}^{\lambda'}_{n+1}=\mathbb{P}(\mathcal{F}_{N+n+1}/\mathcal{S}_n)$.
In particular, $\mathbb{V}^\lambda_{n+1}$ is a sub-bundle of $\mathbb{V}^{\lambda'}_{n+1}$ over the same base space.
As a result, from \cite[B.5.6]{Ful:IT} (for details, see the proof of Lemma \ref{sub-bundle}), there is a regular section of $\mathcal{F}_{N+n+1}/\mathcal{F}_{N+n+1-\lambda_{n+1}}\otimes\mathcal{L}_{n+1}^\vee$ over $\mathbb{V}^{\lambda'}_{n+1}$ whose zero locus is $\mathbb{V}^\lambda$.
As a result \[[\mathbb{V}^\lambda_{n+1}]=c_{top}^T(\mathcal{F}_{N+n+1}/\mathcal{F}_{N+n+1-\lambda_{n+1}}\otimes\mathcal{L}_{n+1}^\vee)=\prod_{j=k-n}^{k-n-1+\lambda_{n+1}}(z_{n+1}+t_j)\in H^*_T(\mathbb{V}^{\lambda'}_{n+1}).\]
Now since both $\mathbb{V}^\emptyset_{n+1}$ and $\mathbb{V}^{\lambda'}_{n+1}$ are projectivizations of the bundle $\mathcal{F}_{N+n+1}/\mathcal{S}_n$ over their respective base spaces $\mathbb{V}^\emptyset_n$ and $\mathbb{V}^\lambda_n$, the induction hypothesis implies that \[[\mathbb{V}^{\lambda'}_{n+1}]=\prod_{i=1}^n\prod_{j=k+1-n}^{k-n+\lambda_i}(z_i+t_j)\in H^*_T(\mathbb{V}^\emptyset_{n+1}).\]
Then pushing forward along the inclusion maps $\iota_1:\mathbb{V}^\lambda_{n+1}\to\mathbb{V}^{\lambda'}_{n+1}$ and $\iota_2:\mathbb{V}^{\lambda'}_{n+1}\to\mathbb{V}^\emptyset_{n+1}$ gives \[\begin{split}[\mathbb{V}^\lambda_{n+1}]\cap[\mathbb{V}^\emptyset_{n+1}] & =(\iota_2)_*([\mathbb{V}^\lambda_{n+1}]\cap[\mathbb{V}^{\lambda'}_{n+1}])\\ & =(\iota_2)_*((\iota_1)_*([\mathbb{V}^\lambda_{n+1}]\cap[\mathbb{V}^\lambda_{n+1}])\cap[\mathbb{V}^{\lambda'}_{n+1}])\\ & =(\iota_2)_*\left(\prod_{j=k-n}^{k-n-1+\lambda_{n+1}}(z_{n+1}+t_j)\cap[\mathbb{V}^{\lambda'}_{n+1}]\right)\\ & =\left(\prod_{i=1}^n\prod_{k+1-i}^{k-i+\lambda_i}(z_i+t_j)\right)\prod_{k-n}^{k-n-1+\lambda_{n+1}}(z_{n+1}+t_j)\\ & = \prod_{i=1}^{n+1}\prod_{k+1-i}^{k-i+\lambda_i}(z_i+t_j).\end{split}\]
This completes the induction step and so completes the proof.
\end{proof}

In order to look at the pushforward of classes in $\mathbb{V}^\emptyset$, we use the fact that for any $\kappa\in H^*_T(\mathbb{V}^\emptyset)$, \[\kappa=\sum_{x\in(\mathbb{V}^\emptyset)^T}\kappa|_x\frac{[x]}{[x]|_x}.\]
We also know that $\pi_*([x])=[\pi(x)]$, and so \[\pi_*(\kappa)=\sum_{x\in(\mathbb{V}^\emptyset)^T}\kappa|_x\frac{[\pi(x)]}{[x]|_x}.\]
For a smooth point $x\in X$, we have $[x]|_x=c^T_{top}(T_xX)$, the Euler class of the tangent space at $x$.
Then localizing this at some fixed point $y\in (\Gr(k,N+k))^T$ gives 
\[\pi_*(\kappa)|_{y}=\sum_{x:\pi(x)=y}\frac{[y]|_y}{[x]|_{x}}\kappa|_{x}.\]
In the Grassmannian, the fixed points are given by the spans of $k$ vectors in the standard basis: \[e_{\lambda}=\langle e_{i_1},...,e_{i_k}\rangle,\]
where for a partition $\lambda$, $i_j=j+\lambda_{k+1-j}$. In $\mathbb{V}^\emptyset$, the fixed points are flags containing vectors in the standard basis and can be parameterized by a partition $\lambda$ and a permutation $w\in S_k$: 
\[e_{\lambda,w}=(\langle e_{i_{w(1)}}\rangle,\langle e_{i_{w(1)}},e_{i_{w(2)}}\rangle,...,\langle e_{i_{w(1)}},...,e_{i_{w(k)}}\rangle).\]
For certain combinations of $\lambda$ and $w$, $S_i\nsubseteq F_{N+i}$ for some $1\leq i\leq k$, meaning $e_{\lambda,w}\notin\mathbb{V}^\emptyset$, so those are not fixed points in $\mathbb{V}^\emptyset$. 
With these definitions, $\pi(e_{\lambda,w})=e_{\lambda}$ for all appropriate $w\in S_k$ and all $\lambda\leq (N^k)$. Now since $z_i=c_1^T(\mathcal{L}_i^\vee)$, we have that \[z_j|_{e_{\lambda,w}}=-t_{i_{w(j)}}.\]
With this, we are able to calculate localizations:
\begin{lem}\label{pfloc}
Given a polynomial $f(z_1,...,z_k)\in\mathbb{Z}[t_1,...,t_{N+k}][z_1,...,z_n]$, representing a class in $H^*_T(\mathbb{V}^\emptyset)$ and a partition $\lambda\leq (N^k)$, the localization of $\pi_*(f)$ at the fixed point corresponding to $\lambda$ in $\Gr(k,N+k)$ is
\[\pi_*(f(z_1,...,z_k)\cap[\mathbb{V}^\emptyset])|_{e_\lambda}=\sum_{w\in S_k}\left(\prod_{i=1}^{k-1}\prod_{j=i+1}^k\frac{(t_{j-i}-t_{i_{w(i)}})}{(t_{i_{w(j)}}-t_{i_{w(i)}}))}\right)f(-t_{i_{w(1)}},...,-t_{i_{w(k)}}),\]
where $i_j=j+\lambda_{k+1-j}$ for $1\leq j\leq k$.
\end{lem}
\begin{proof}
Since the localizations of the fixed points are given by the Euler class of the tangent bundle, we look at the tangent bundle in each space.
The tangent bundle in the Grassmannian is given by $T_{\Gr(k,N+k)}=\mathcal{S}^\vee\otimes\mathcal{Q}$, where $S$ is the tautological sub-bundle and $Q$ is the tautological quotient bundle. To find the localization of the Euler class of this, we look at the tangent space at a fixed point $\lambda$ and see how $T$ acts on it. In particular, we know that the Chern roots of $\mathcal{S}^\vee|_{e_\lambda}$ are given by $-t_{i_{j}}$ for $1\leq j\leq k$, and the Chern roots of $\mathcal{Q}$ are given by $t_j$ for $j\neq i_p$ for any $1\leq p\leq k$. With this, we have
\[[e_{\lambda}]|_{e_\lambda}=c_{top}^T(\mathcal{S}^\vee\otimes\mathcal{Q})|_{e_\lambda}=\prod_{i\in J}\prod_{j\notin J}(t_j-t_i),\] where $J=\{j+\lambda_{k+1-j}:1\leq j\leq k\}$.

The tangent bundle in $\mathbb{V}^\emptyset$ is given by $T_{\mathbb{V}^\emptyset}=\oplus_{i=1}^k(\mathcal{L}_i^\vee\otimes\mathcal{Q}_i)$, since it is a tower of projective bundles and the relative tangent bundle for $p:\mathbb{V}_i^\emptyset\to\mathbb{V}_{i-1}^\emptyset$ is $\mathcal{L}_i^\vee\otimes\mathcal{Q}_i$ (see \cite[B.5.8]{Ful:IT}).
At some fixed point $e_{\lambda,w}=(S_1,...,S_k)$, we have that $\mathcal{L}_i|_{e_{\lambda,w}}=S_i/S_{i-1}$ and $\mathcal{Q}_i|_{e_{\lambda,w}}=F_{N+i}/S_i$.
As a result, we have that \[[e_{\lambda,w}]|_{e_{\lambda,w}}=c_{top}^T(\oplus_{i=1}^k(\mathcal{L}_i^\vee\otimes\mathcal{Q}_i))|_{e_{\lambda,w}}=\prod_{i=1}^k\prod_{j\in J_i}(t_j-t_{i_{w(i)}}),\]
where $J_i=\{j:k+1-i\leq j\leq N+k\}\backslash\{i_{w(j)},1\leq j\leq i\}$, and as before $i_j=j+\lambda_{k+1-j}$. Here note that $J_i$ must have exactly $N$ elements, meaning that $k+1-i\leq i_{w(j)}\leq N+k$ for all $1\leq j\leq i$ is required for this formula to be correct. This condition is equivalent to the condition that $S_i\subseteq F_{N+i}$, though, which is the condition required for $e_{\lambda,w}$ to be in $\mathbb{V}^\emptyset$, so this will be correct for all valid combinations of $w$ and $\lambda$.

To better cancel out terms, we can rewrite the product in $[e_{\lambda}]|_{e_\lambda}$ as \[[e_\lambda]|_{e_\lambda}=\prod_{i=1}^k\prod_{j\in J_k}(t_j-t_{i_{w(i)}}),\] since $i_{w(i)}$ will run across all of the indices for the basis vectors in S, and $J_k$ was defined to be all the other basis vectors. With this, we have that
\[\frac{[e_\lambda]|_{e_\lambda}}{[e_{\lambda,w}]|_{e_\lambda,w}}=\prod_{i=1}^k\frac{\prod_{j\in J_k\backslash J_i}(t_j-t_{i_{w(i)}})}{\prod_{j\in J_i\backslash J_k}(t_j-t_{i_{w(i)}})}\]

To simplify this further, we can specialize to the case where $\lambda$ satisfies the condition that $e_{\lambda,w}$ is in $\mathbb{V}^\emptyset$ for all $w\in S_k$, which in particular means that $i_j$ satisfies $k\leq i_j\leq N+k$ for all $1\leq j\leq k$. As a result, $J_k\backslash J_i=\{1,...,k-i\}$ and $J_i\backslash J_k=\{i_{w(k)},...,i_{w(i+1)}\}$. With this, we have that \[\frac{[e_\lambda]|_{e_\lambda}}{[e_{\lambda,w}]|_{e_{\lambda,w}}}=\prod_{i=1}^{k-1}\prod_{j=i+1}^k\frac{(t_{j-i}-t_{i_{w(i)}})}{(t_{i_{w(j)}}-t_{i_{w(i)}})}.\]

With this formula, it happens that if $e_{\lambda,w}\notin\mathbb{V}^\emptyset$, the numerator vanishes. In particular, $e_{\lambda,w}\notin\mathbb{V}^\emptyset$ exactly when there exists an $1\leq i\leq k$ such that $w(i)<k+1-i$. In such a case, the product $\prod_{j=i+1}^k(t_{j-i}-t_{i_{w(i)}})$ vanishes because the $t_{j-i}$ ranges from 1 to $k-i$, which is the possible range for $w(i)$. As a result, we can take the sum over all permutations $w\in S_k$, even if $e_{\lambda,w}\notin\mathbb{V}^\emptyset$, and it will still be correct.

Then we have \[\begin{split}\pi_*(f(z_1,...,z_k)\cap[\mathbb{V}^\emptyset])|_{e_\lambda} & =\sum_{e_{\lambda,w}\in\mathbb{V}^\emptyset}\frac{[e_\lambda]|_{e_\lambda}}{[e_{\lambda,w}]|_{e_{\lambda,w}}}f(-t_{i_{w(1)}},...,-t_{i_{w(k)}})\\ & =\sum_{e_{\lambda,w}\in\mathbb{V}^\emptyset}\prod_{i=1}^{k-1}\prod_{j=i+1}^k\frac{(t_{j-i}-t_{i_{w(i)}})}{(t_{i_{w(j)}}-t_{i_{w(i)}})}f(-t_{i_{w(1)}},...,-t_{i_{w(k)}})\\ & =\sum_{w\in S_k}\prod_{i=1}^{k-1}\prod_{j=i+1}^k\frac{(t_{j-i}-t_{i_{w(i)}})}{(t_{i_{w(j)}}-t_{i_{w(i)}})}f(-t_{i_{w(1)}},...,-t_{i_{w(k)}}).\end{split}\]

\end{proof}

If we define $x_1,..,x_k$ to be the Chern roots of the dual to the tautological sub-bundle, $\mathcal{S}^\vee$, in the Grassmannian, then for any symmetric polynomial $p(x_1,...,x_k)\in H^*_T(\Gr(k,N+k))$, \[p(x_1,...,x_k)|_{e_\lambda}=p(-t_{i_1},...,-t_{i_k}),\] where $i_j=j+\lambda_{k+1-j}$ for $1\leq j\leq k$.
As a result we can modify the localization formula from the above section to express $\pi_*(f(z_1,...,z_k)\cap[\mathbb{V}^\emptyset])$ as a symmetric polynomial in $\mathbb{Z}[t_1,...,t_{N+k}][x_1,...,x_k]^{S_k}$:
\begin{lem}\label{pfpoly}
Given a polynomial $f(z_1,...,z_k;t)\in\mathbb{Z}[t_1,...,t_{N+k}][z_1,...,z_n]$, \[\pi_*(f(z_1,...,z_k;t)\cap[\mathbb{V}^\emptyset])=\sum_{w\in S_k}\left(\prod_{i=1}^{k-1}\prod_{j=i+1}^k\frac{(t_{j-i}+x_{w(i)})}{(x_{w(i)}-x_{w(j)}))}\right)f(x_{w(1)},...,x_{w(k)};t).\]
\end{lem}
\begin{proof}
By localizing both sides of the equation at any fixed point, Lemma \ref{pfloc} gives the same localizations. Since the two classes give the same localization at every fixed point, they are the same class by injectivity of the localization map \cite[Theorem 1.2.2]{GKM}.
\end{proof}

\begin{remark}
This pushforward can also be expressed in terms of BGG operators. See Proposition \ref{pushforwardoperator}.
\end{remark}

Note that while the expression contains fractions, it is a polynomial. This can be seen by the fact that for any $w\in S_k$, \[\prod_{i=1}^k\prod_{j=i+1}^k(x_{w(i)}-x_{w(j)})=\text{sgn}(w)\prod_{i=1}^{k-1}\prod_{j=i+1}^k(x_i-x_j)=\text{sgn}(w)\prod_{1\leq i<j\leq k}(x_i-x_j).\]
As a result the expression becomes \[\prod_{1\leq i<j\leq k}\frac{1}{x_i-x_j}\sum_{w\in S_k}\text{sgn}(w)\left(\prod_{i=1}^k\prod_{j=i+1}^k(t_{j-i}+x_{w(i)})\right)f(x_{w(1)},...,x_{w(k)}).\]
Since the numerator is an alternating sum over $S_k$, it is a skew-symmetric polynomial, which means that specializing $x_i=x_j$ for any $i\neq j$ results in 0. This implies $x_i-x_j$ divides it. As a result, the denominator divides the numerator, and so the expression is a polynomial, which will be symmetric because it is the quotient of two skew-symmetric polynomials.

\section{Factorial Schur functions}

In this section we will use the results of the previous section to conclude with the main result for cohomology (Theorem \ref{factSchur} and Corollary \ref{srulealg}). Specifically, this is the new method by which we re-prove the result that factorial Schur functions represent the equivariant Schubert classes in Grassmannians.

The factorial Schur function is given by \cite[eq. (6.3), (6.4)]{MacD:Schur}: \[s_\lambda(x|a)=\frac{det((x_i|a)^{\lambda_j+\delta_j})_{1\leq i,j\leq k}}{\prod_{i<j}(x_i-x_j)},\] where $\delta=(k-1,k-2,...,1,0)$, $a$ is a sequence $(a_i)_{i\in\mathbb{Z}}$, and $(x_i|a)^m=\prod_{j=1}^m(x_i+a_j)$.
For our purposes, in place of the sequence $a$ we use the sequence $t=(t_i)$, with $t_i=0$ for $i\leq 0$ and $i\geq N+k+1$ and the usual $t_i=c^T_1(\langle e_i\rangle)$.

As the main result in cohomology for this paper, we can recover the factorial Schur functions by pushing forward the class of $\mathbb{V}^\lambda$:

\begin{thm}\label{factSchur}
For any partition $\lambda\leq(N^k)$, \[[\Omega^\lambda]=\pi_*([\mathbb{V}^\lambda]\cap[\mathbb{V}^\emptyset])=s_\lambda(x|t).\]
\end{thm}
\begin{proof}
From the fact that $\pi$ is a surjective, birational morphism, $\pi_*([\mathbb{V}^\lambda]\cap[\mathbb{V}^\emptyset])=[\Omega^\lambda]$.
Then from Lemma \ref{dubprod}, we have that \[[\mathbb{V}^\lambda]=\prod_{i=1}^k\prod_{k+1-i}^{k-i+\lambda_i}(z_i+t_j).\]
Then we rewrite $\prod_{j=i+1}^k(t_{j-i}+x_{w(i)})=\prod_{j=1}^{k-i}(t_{j}+x_{w(i)})$, which gives that \[\left(\prod_{i=1}^k\prod_{j=1}^{k-i}(t_j+x_{w(i)})\right)\prod_{i=1}^k\prod_{j=k+1-i}^{k-i+\lambda_i}(t_j+x_{w(i)})=\prod_{i=1}^k\prod_{j=1}^{k-i+\lambda_i}(t_j+x_{w(i)}).\]
As a result, plugging $f=\prod_{i=1}^k\prod_{j=k+1-i}^{k-i+\lambda_i}(z_i+t_j)$ into Lemma \ref{pfpoly} and simplifying gives \[\pi_*(\prod_{i=1}^k\prod_{j=k+1-i}^{k-i+\lambda_i}(z_i+t_j)\cap[\mathbb{V}^\emptyset])=\prod_{1\leq i<j\leq k}\frac{1}{x_i-x_j}\sum_{w\in S_k}\text{sgn}(w)\prod_{i=1}^k\prod_{j=1}^{k-i+\lambda_i}(t_j+x_{w(i)}).\]
Then using $\prod_{j=1}^{k-i+\lambda_i}(x_{w(i)}+t_j)=(x_{w(i)}|-t)^{k-i+\lambda_i}=(x_{w(i)}|-t)^{\delta_i+\lambda_i}$, we obtain \[\pi_*(\prod_{i=1}^k\prod_{j=k+1-i}^{k-i+\lambda_i}(z_i+t_j)\cap[\mathbb{V}^\emptyset])=\prod_{1\leq i<j\leq k}\frac{1}{x_i-x_j}\sum_{w\in S_k}\text{sgn}(w)\prod_{i=1}^k(x_{w(i)}|-t)^{\delta_i+\lambda_i}.\]
By the definition of the determinant $\text{det}(A)=\sum_{w\in S_k}\text{sgn}(w)\prod_{i=1}^ka_{w(i),i}$, we have \[\pi_*(\prod_{i=1}^k\prod_{j=k+1-i}^{k-i+\lambda_i}(z_i+t_j)\cap[\mathbb{V}^\emptyset])=\frac{\text{det}((x_i|-t)^{\lambda_j+\delta_j})_{1\leq i,j\leq k}}{\prod_{1\leq i<j\leq k}(x_i-x_j)}=s_\lambda(x|-t).\]
\end{proof}

\begin{cor}[The straightening rule]\label{srulealg}
For any sequence of nonnegative integers $\mu=(\mu_1,...,\mu_k)$, let \[p_\mu(z,t)=\prod_{i=1}^k\prod_{j=k+1-i}^{k-i+\mu_i}(z_i+t_j).\]
Then
\begin{enumerate}
\item If $\mu$ is a partition, then $\pi_*(p_\mu(z,t)\cap[\mathbb{V}^\emptyset])=s_\mu(x|t)$.

\item If $\mu$ is not a partition and the values $\mu_i+k-i$ are all distinct, then there exists a permutation $w$ and a partition $\lambda$ such that $\lambda+\delta=(\mu_{w(1)}+\delta_{w(1)},...,\mu_{w(k)}+\delta_{w(k)})$, and $\pi_*(p_\mu(z,t)\cap[\mathbb{V}^\emptyset])=\text{sgn}(w)s_\lambda(x|t)$.

\item If $\mu$ is not a partition and the values $\mu_i+k-i$ are not all distinct, then $\pi_*(p_\mu(z,t)\cap[\mathbb{V}^\emptyset])=0$.
\end{enumerate}
\end{cor}
\begin{proof}
Claim (1) is Theorem \ref{factSchur}.
By the same process, we have that for any sequence $\mu$, \[\pi_*(p_\mu(z,t)\cap[\mathbb{V}^\emptyset])=\frac{\text{det}((x_i|t)^{\mu_j+\delta_j})_{1\leq i,j\leq k}}{\prod_{1\leq i<j\leq k}(x_i-x_j)}.\]
From the properties of the determinant, applying a permutation $w$ to the columns of the matrix $((x_i|t)^{\mu_j+\delta_j})_{1\leq i,j\leq k}$ will multiply the determinant of $\text{sgn}(w)$.
In particular, if $\lambda+\delta=w(\mu+\delta)$, then \[\text{det}((x_i|t)^{\mu_j+\delta_j})_{1\leq i,j\leq k}=\text{sgn}(w)\text{det}((x_i|t)^{\lambda_i+\delta_i})_{1\leq i,j\leq k}.\]
If such a permutation $w$ and a partition $\lambda$ exist, this proves claim (2).
Assuming the values $\mu_i+\delta_i$ are all distinct, they can be arranged in descending order, and as a result there exists a permutation $w$ such that $(\mu_{w(1)}+\delta_{w(1)},...,\mu_{w(k)}+\delta_{w(k)})$ is in descending order.
Therefore $\mu_{w(i)}+\delta_{w(i)}\geq \mu_{w(i+1)}+\delta_{w(i+1)}+1$, and as a result $\mu_{w(i)}+\delta_{w(i)}-(k-i)\geq \mu_{w(i+1)}+\delta_{w(i+1)}-(k-(i+1))$.
So then $\lambda=(\lambda_1,...,\lambda_k)$ with $\lambda_i=\mu_{w(i)}+\delta_{w(i)}-(k-i)$ for $1\leq i\leq k$ is a partition and $\lambda_i+\delta_i=\mu_{w(i)}+\delta_{w(i)}$, as required.
For claim (3), note that if two values of $\mu_i+\delta_i$ are not distinct, then two columns in the matrix are identical, which implies the determinant is 0.
\end{proof}

Claims (2) and (3) together are known as the straightening rule.

\begin{exmps}\ \\
\begin{enumerate}
\item $\pi_*(p_{(1,2)}(z,t)\cap[\mathbb{V}^\emptyset])=0$, since in this case $\mu=(1,2)$, and $\mu+\delta=(1+1,2+0)=(2,2)$, which does not have distinct parts.

\item $\pi_*(p_{(0,2)}(z,t)\cap[\mathbb{V}^\emptyset])=-s_{(1,1)}(x|-t)$, since in this case $\mu=(0,2)$ and $\mu+\delta=(0,2)+(1,0)=(1,2)$, which can be permuted into $\lambda+\delta=(2,1)$, which results in $\lambda=(2,1)-(1,0)=(1,1)$. Since the permutation required to do this is odd, there is a minus sign.
\end{enumerate}
\end{exmps}

\section{K-Theory Preliminaries}

\subsection{K-Theory}
We recall some basic facts about K-theory, using \cite{Brion:flag} as a reference.
The Grothendieck ring $K^0(X)$ is generated by equivalence classes of vector bundles subject to the relation that whenever there is an exact sequence  \[0\to E_1\to E_2\to E_3\to0,\] the classes of those vector bundles satisfy $[E_1]+[E_3]=[E_2]$. The multiplication of classes is given by the class of the tensor product of the vector bundles.
In smooth varieties, a coherent sheaf can be resolved by finitely many vector bundles, and so the class of a coherent sheaf can be expressed in $K^0(X)$ as an alternating sum classes of vector bundles.
Any morphism of schemes $f:X\to Y$ induces a pullback map $f^*:K^0(Y)\to K^0(X)$ by $f^*([E])=[f^*E]$.
Since any vector bundle has a sheaf of sections, it can be seen as a sheaf, and so a pushforward can be defined for any proper morphism $f:X\to Y$, $f_*:K^0(X)\to K^0(Y)$ by $f_*([E])=\sum_{j\geq 0}(-1)^j[R^jf_*(E)]$.
There is also a projection formula $f_*((f^*\alpha)\cdot\beta)=\alpha\cdot f_*\beta$ \cite[Section 3.3]{Brion:flag}.
Given a vector bundle $V$, the $\lambda_y$ class of $V$ is given by \[\lambda_y(V)=\sum_{p\geq0}[\Lambda^pV]\cdot y^p.\]
This has the property that for any exact sequence of vector bundles \[0\to V_1\to V_2\to V_3\to 0,\] we have that $\lambda_y(V_2)=\lambda_y(V_1)\lambda_y(V_3)$ \cite{Hir:K}.

For equivariant K-theory, the equivariant Grothendieck ring $K_T(X)$ is instead generated by equivalence classes of equivariant vector bundles or equivariant coherent sheaves. For equivariant morphisms, there are similarly pushforward and pullbacks maps on the equivariant Grothendieck rings \cite[Chapter 5.2]{CG:eqK}.

\subsection{Equivariant localization}
Like in cohomology, for each $T$-fixed point $x\in X^T$, the inclusion map $\iota_x:\{x\}\to X$ induces a pullback map on equivariant K-theory $\iota_x^*:K_T(X)\to K_T(x)$. 
For $\alpha\in K_T(X)$, define the localization of $\alpha$ at the fixed point $x$ as $\alpha|_x=\iota_x^*(\alpha)$.
The equivariant K-theory of a point is the representation ring of $T$, which is isomorphic to the Laurent polynomial ring $\mathbb{Z}[e^{\pm t_1},...,e^{\pm t_{N+k}}]$, where $e^{t_i}$ are the characters corresponding to a basis of the Lie algebra of $T$ and $e^{t_i}=[\mathbb{C}_{t_i}]$ \cite[5.2.1]{CG:eqK}. 
When there are finitely many fixed points, the localization map $\iota^*:K_T(X)\to K_T(X^T)=\oplus_{x\in X^T}K_T(x)$ is injective by the localization theorem, see \cite{Niel:eqK}, so a class can be determined by its localizations. 
As a result, since the structure sheaf of a fixed point is only supported on that fixed point, we can express any $\kappa\in K_T(X)$ in terms of the structure sheaves of the fixed points by \[\kappa=\sum_{x\in X^T}\frac{\kappa|_x}{[\mathcal{O}_x]|_x}[\mathcal{O}_x].\]
For a smooth point $x\in X$, we have $[\mathcal{O}_x]|_x=\lambda_{-1}(T^*_xX)$.
Then since $\pi_*([\mathcal{O}_x])=[\mathcal{O}_{\pi(x)}]$, we can use this to calculate the localizations of the pushforward of a class: \[\pi_*(\kappa)|_x=\sum_{y:\pi(y)=x}\frac{[\mathcal{O}_x]|_x}{[\mathcal{O}_y]|_y}\kappa|_y.\]

\section{Calculation of the K-theoretic pushforward}

This section calculates the class of the structure sheaf of $\mathbb{V}^\lambda$ on $\mathbb{V}^\emptyset$ (Lemma \ref{Kpf}) as well as the localizations of the pushforward of any class (Lemma \ref{Kpfloc}).

In the Bott-Samelson varieties, the fixed points are \[e_{\lambda,w}=(\langle e_{i_{w(1)}}\rangle,\langle e_{i_{w(1)}},e_{i_{w(2)}}\rangle,...,\langle e_{i_{w(1)}},...,e_{i_{w(k)}}\rangle),\] where $i_j=j+\lambda_{k+1-j}$, for appropriate pairs $(\lambda,w)$, where $\lambda=(\lambda_1,...,\lambda_k)$ is a partition with $\lambda_1\leq N$ and $w\in S_k$.
In $\Gr(k,N+k)$, the fixed points are $e_\lambda=\langle e_{i_1},...,e_{i_k}\rangle$, where $i_j=j+\lambda_{k+1-j}$.
Under the map $\pi:\mathbb{V}^\emptyset\to\Gr(k,N+k)$, $\pi(e_{\lambda,w})=e_\lambda$ for all $w\in S_k$ such that $e_{\lambda,w}\in\mathbb{V}^\emptyset$.
Then, similarly to what was done in Lemma \ref{pfloc}, using the fact that $[\mathcal{O}_{e_\lambda}]|_{e_\lambda}=\lambda_{-1}(T^*_{e_\lambda}\Gr(k,N+k))$ and $[\mathcal{O}_{e_{\lambda,w}}]|_{e_{\lambda,w}}=\lambda_{-1}(T^*_{e_{\lambda,w}}\mathbb{V}^\emptyset)$, we obtain
\begin{lem}\label{Kpfloc}
Given a Laurent polynomial in the $\mathcal{L}_i$'s, \[f(\mathcal{L}_1^{\pm1},...,\mathcal{L}_k^{\pm1};e^{\pm t})\in\mathbb{Z}[e^{\pm t_1},...,e^{\pm t_{N+k}}][\mathcal{L}_1^{\pm 1},...,\mathcal{L}_k^{\pm1}]\] representing a class in $K_T(\mathbb{V}^\emptyset)$ and a partition $\lambda\leq(N^k)$, the localization of $\pi_*(f)$ at the fixed point $e_\lambda$ in $\Gr(k,N+k)$ is \[\pi_*(f(\mathcal{L}_1^{\pm1},...,\mathcal{L}_k^{\pm1}))|_{e_{\lambda}}=\sum_{w\in S_k}\left(\left(\prod_{i=1}^{k-1}\prod_{j=i+1}^k\frac{(1-e^{t_{i_{w(i)}}-t_{j-i}})}{(1-e^{t_{i_{w(i)}}-t_{i_{w(j)}}})}\right)f(e^{\mp t_{i_{w(1)}}},...,e^{\mp t_{i_{w(k)}}};e^{\pm t})\right)\]
for $i_j=j+\lambda_{k+1-j}$.
\end{lem}
\begin{proof}
Since $[\mathcal{O}_{e_\lambda}]|_{e_\lambda}=\lambda_{-1}(T_{e_\lambda}^\vee(\Gr(k,N+k)))$ and $[\mathcal{O}_{e_{\lambda,w}}]|_{e_{\lambda,w}}=\lambda_{-1}(T_{e_{\lambda,w}}^\vee(\mathbb{V}^\emptyset))$, we have that \[\frac{[\mathcal{O}_{e_{\lambda}}]|_{e_\lambda}}{[\mathcal{O}_{e_{\lambda,w}}]|_{e_{\lambda,w}}}=\frac{\prod_{i\in J}\prod_{j\notin J}\lambda_{-1}(\mathbb{C}_{t_i}\otimes\mathbb{C}^\vee_{t_j})}{\prod_{i=1}^k\prod_{j\in J_i}\lambda_{-1}(\mathbb{C}_{t_{i_{w(i)}}}\otimes\mathbb{C}^\vee_{t_j})},\] where $J=\{i_j:1\leq j\leq k\}$ and $J_i=\{j:k+1-i\leq j\leq N+k\}\backslash\{i_{w(j)},1\leq j\leq i\}$ with $i_j=j+\lambda_{k+1-j}$, as in the proof of Lemma \ref{pfloc}.
Then with the same cancellations and simplifications from the proof of Lemma \ref{pfloc}, we obtain \[\frac{[\mathcal{O}_{e_{\lambda}}]|_{e_\lambda}}{[\mathcal{O}_{e_{\lambda,w}}]|_{e_{\lambda,w}}}=\prod_{i=1}^{k-1}\prod_{j=i+1}^k\frac{\lambda_{-1}(\mathbb{C}_{t_{i_{w(i)}}}\otimes\mathbb{C}^\vee_{t_j})}{\lambda_{-1}(\mathbb{C}_{t_{i_{w(i)}}}\otimes\mathbb{C}_{t_{i_{w(j)}}}^\vee)}.\]
Then from the fact that $\lambda_{-1}(\mathbb{C}_{t_i}\otimes\mathbb{C}^\vee_{t_j})=1-[\mathbb{C}_{t_i}\otimes\mathbb{C}^\vee_{t_j}]=1-e^{t_i-t_j}$, we obtain \[\frac{[\mathcal{O}_{e_{\lambda}}]|_{e_\lambda}}{[\mathcal{O}_{e_{\lambda,w}}]|_{e_{\lambda,w}}}=\prod_{i=1}^{k-1}\prod_{j=i+1}^k\frac{(1-e^{t_{i_{w(i)}}-t_{j-i}})}{(1-e^{t_{i_{w(i)}}-t_{i_{w(j)}}})}.\]
Then we have \[\begin{split}\pi_*(f(\mathcal{L}_1^{\pm1},...,\mathcal{L}_k^{\pm1}))|_x & =\sum_{e_{\lambda,w}\in\mathbb{V}^\emptyset}\frac{[\mathcal{O}_{e_\lambda}]|_{e_\lambda}}{[\mathcal{O}_{e_{\lambda,w}}]|_{e_{\lambda,w}}}f(e^{\mp t_{i_{w(1)}}},...,e^{\mp t_{i_{w(k)}}};e^{\pm t})\\ & =\sum_{e_{\lambda,w}\in\mathbb{V}^\emptyset}\prod_{i=1}^{k-1}\prod_{j=i+1}^k\frac{(1-e^{t_{i_{w(i)}}-t_{j-i}})}{(1-e^{t_{i_{w(i)}}-t_{i_{w(j)}}})}f(e^{\mp t_{i_{w(1)}}},...,e^{\mp t_{i_{w(k)}}};e^{\pm t})\\ & =\sum_{w\in S_k}\prod_{i=1}^{k-1}\prod_{j=i+1}^k\frac{(1-e^{t_{i_{w(i)}}-t_{j-i}})}{(1-e^{t_{i_{w(i)}}-t_{i_{w(j)}}})}f(e^{\mp t_{i_{w(1)}}},...,e^{\mp t_{i_{w(k)}}};e^{\pm t}).\end{split}\]
\end{proof}

To calculate $[\mathcal{O}_{\mathbb{V}^\lambda}]$, we need to use some facts from \cite[Appendix B]{Ful:IT} that can be summarized as the following Lemma.
\begin{lem}\label{sub-bundle}
Let $X$ be a projective variety with $p:E,F\to X$ vector bundles and $E\subseteq F$. Then for the quotient bundle $G=F/E$, $[\mathcal{O}_{\mathbb{P}(E)}]=\lambda_{-1}((\mathcal{O}_{\mathbb{P}(F)}(1)\otimes p^*(G))^\vee)\in K_T(\mathbb{P}(F))$
\end{lem}
\begin{proof}
From the tautological sequence on $\mathbb{P}(F)$, there is a vector bundle map $\mathcal{O}_{\mathbb{P}(F)}(-1)\to p^*F$.
When composed with the quotient map $p^*F\to p^*G$, this becomes a map $\mathcal{O}_{\mathbb{P}(F)}(-1)\to p^*G$, and this map is the $0$ map exactly when the fiber of $\mathcal{O}_{\mathbb{P}(F)}(-1)$ is included in the fiber of $p^*E$, which happens when $x\in\mathbb{P}(E)\subseteq\mathbb{P}(F)$.
This map corresponds to a section $s\in\text{Hom}(\mathcal{O}_{\mathbb{P}(F)}(-1),p^*(G))\cong\mathcal{O}_{\mathbb{P}(F)}(1)\otimes p^*G$.
This section is regular and its zero locus is $Z(s)=\mathbb{P}(E)$ \cite[B.5.6]{Ful:IT}.
Then by \cite[B.3.4, (*)]{Ful:IT}, there is an exact sequence \[0\to\Lambda^r(\mathcal{O}_{\mathbb{P}(F)}(1)\otimes p^*G)^\vee\to...\to\Lambda(\mathcal{O}_{\mathbb{P}(F)}(1)\otimes p^*G)^\vee\to\mathcal{O}_{\mathbb{P}(F)}\to\mathcal{O}_{\mathbb{P}(E)}\to0.\]
This implies \[[\mathcal{O}_{\mathbb{P}(E)}]-[\mathcal{O}_{\mathbb{P}(F)}]+\sum_{i=1}^r(-1)^{i-1}[\Lambda^i(\mathcal{O}_{\mathbb{P}(F)}(1)\otimes p^*G)^\vee]=0.\]
Then solving for $[\mathcal{O}_{\mathbb{P}(E)}]$ gives \[\mathcal{O}_{\mathbb{P}(E)}]=\sum_{i=0}^r(-1)^i[\Lambda^i(\mathcal{O}_{\mathbb{P}(F)}(1)\otimes p^*G)^\vee]=\lambda_{-1}((\mathcal{O}_{\mathbb{P}(F)}(1)\otimes p^*G)^\vee).\]
\end{proof}
\begin{lem}\label{Kpf}
For a partition $\lambda=(\lambda_1,...,\lambda_k)$ with $\lambda_1\leq N$, any class $\alpha\in K_0(\mathbb{V}^\lambda)$, and the inclusion map $\iota:\mathbb{V}^\lambda\to\mathbb{V}^\emptyset$, \[[\mathcal{O}_{\mathbb{V}^\lambda}]=\prod_{i=1}^k\prod_{j=k+1-i}^{k-i+\lambda_i}(1-[\mathcal{L}_i\otimes\mathbb{C}_{t_j}^\vee])\in K_T(\mathbb{V}^\emptyset).\]
\end{lem}
\begin{proof}
We proceed by induction on $n$ in $\mathbb{V}^\lambda_n$.
The base case, n=0, is a point, which cannot have nonempty, proper subvarieties, so this holds trivially.
Suppose \[[\mathcal{O}_{\mathbb{V}_n^\lambda}]=\prod_{i=1}^n\prod_{j=k+1-i}^{k-i+\lambda_i}(1-[\mathcal{L}_i\otimes\mathbb{C}^\vee_{t_j}])\in K_T(\mathbb{V}^\emptyset_n).\]
Recall $\mathbb{V}^\lambda_n=\{(S_1,...,S_n): \dim(S_i)=i, S_i\subseteq F_{N+i-\lambda_i}\}$, and $\mathbb{V}^\lambda_{n+1}=\mathbb{P}(\mathbb{F}_{N+n+1-\lambda_{n+1}}/(\mathcal{L}_1\oplus...\oplus\mathcal{L}_n))$, where $\mathcal{F}_{N+n+1-\lambda_{n+1}}$ is the vector bundle on $\mathbb{V}^\lambda_n$ whose fiber is $F_{N+n+1-\lambda_{n+1}}$, and $\mathcal{L}_i$ is the vector bundle whose fiber is $S_i/S_{i-1}$.
Then for $\lambda'=(\lambda_1,...,\lambda_n,0)$, we have that $\mathbb{V}^\lambda_{n+1}=\mathbb{P}(\mathcal{F}_{N+n+1-\lambda_{n+1}}/(\mathcal{L}_1\oplus...\oplus\mathcal{L}_n))$ and $\mathbb{V}^{\lambda'}_{n+1}=\mathbb{P}(\mathcal{F}_{N+n+1}/(\mathcal{L}_1\oplus...\oplus\mathcal{L}_n)$.
Since $\mathcal{F}_{N+n+1-\lambda_{n+1}}/(\mathcal{L}_1\oplus...\oplus\mathcal{L}_n)\subseteq\mathcal{F}_{N+n+1}/(\mathcal{L}_1\oplus...\oplus\mathcal{L}_n)$ with quotient $G=\mathcal{F}_{N+n+1}/\mathcal{F}_{N+n+1-\lambda_{n+1}}$, by Lemma \ref{sub-bundle}, we have $[\mathcal{O}_{\mathbb{V}^{\lambda}_{n+1}}]=\lambda_{-1}((G\otimes\mathcal{L}_{n+1}^\vee)^\vee)\in K_T(\mathbb{V}^{\lambda'}_{n+1})$.
Since $F_{N+n+1}/F_{N+n+1-\lambda_{n+1}}=\langle e_{N+k+1-(N+n+1)},...,e_{N+k-(N+n+1-\lambda_{n+1})}\rangle$, this means \[[\mathcal{O}_{\mathbb{V}^\lambda_{n+1}}]=\prod_{j=k-n}^{k-n-1+\lambda_{n+1}}(1-\mathbb{C}_{t_j}^\vee\otimes\mathcal{L}_{n+1}).\]
Then since $\mathbb{V}^\lambda_n\subseteq\mathbb{V}^\emptyset_n$, and both $\mathbb{V}^{\lambda'}_{n+1}$ and $\mathbb{V}^\emptyset_{n+1}$ are projectivizations of the bundle $E=\mathcal{F}_{N+n+1}/(\mathcal{L}_1\oplus...\oplus\mathcal{L}_n)$ over the previous spaces, $\mathbb{V}^{\lambda'}_{n+1}\subseteq\mathbb{V}^\emptyset_{n+1}$:
\[\begin{tikzcd}
\mathbb{V}^{\lambda'}_{n+1}=\mathbb{P}(E) \arrow[d,"p"] \arrow[r,"\iota_2"] & \mathbb{V}^\emptyset_{n+1}=\mathbb{P}(E) \arrow[d,"p"]\\
\mathbb{V}^{\lambda}_n \arrow[r,"\iota_2"] & \mathbb{V}^\emptyset_n
\end{tikzcd}\]
Furthermore, $[\mathbb{V}^{\lambda'}_{n+1}]=p^*([\mathbb{V}^{\lambda}_n])\in K_T(\mathbb{V}^\emptyset_{n+1})$, where $p:\mathbb{V}^\emptyset_{n+1}\to\mathbb{V}^\emptyset_n$ is the projection map.
Now for the inclusion maps $\iota_1:\mathbb{V}^\lambda_{n+1}\to\mathbb{V}^{\lambda'}_{n+1}$ and $\iota_2:\mathbb{V}^{\lambda'}_{n+1}\to\mathbb{V}^\emptyset_{n+1}$, we have $\iota=\iota_2\circ\iota_1$, and so we have \[\begin{split}\iota_*([\mathcal{O}_{\mathbb{V}^\lambda_{n+1}}])& =\iota_{2,*}(\iota_{1,*}([\mathcal{O}_{\mathbb{V}^\lambda_{n+1}}]))\\
& =\iota_{2,*}([\mathcal{O}_{\mathbb{V}^{\lambda'}_{n+1}}]\iota_2^*(\prod_{j=k-n}^{k-n-1+\lambda_{n+1}}(1-\mathbb{C}_{t_j}^\vee\otimes\mathcal{L}_{n+1})))\\
& = \left(\prod_{i=1}^n\prod_{j=k+1-i}^{k-i+\lambda_i}(1-[\mathcal{L}_i\otimes\mathbb{C}^\vee_{t_j}])\in K_T(\mathbb{V}^\emptyset_n)\right)\prod_{j=k-n}^{k-n-1+\lambda_{n+1}}(1-\mathbb{C}_{t_j}^\vee\otimes\mathcal{L}_{n+1})\\
& = \prod_{i=1}^{n+1}\prod_{j=k+1-i}^{k-i+\lambda_i}(1-[\mathcal{L}_i\otimes\mathbb{C}_{t_j}^\vee])\in K_T(\mathbb{V}^\emptyset_{n+1}),\end{split}\]
from the induction hypothesis.
This completes the induction step, and so completes the proof.
\end{proof}

\section{Factorial Grothendieck polynomials}

This section establishes the correspondence between the classes of vector bundles and the variables used in polynomials, uses that correspondence to express the pushforward of any class in terms of polynomials (Lemma \ref{Kpfpoly}), and uses that to get the main result in K-theory (Theorem \ref{factGrothendieck} and Corollary \ref{SRuleK}).

\subsection{Grothendieck polynomials} 
Grothendieck polynomials represent the classes of the structure sheaves of Schubert varieties in flag varieties \cite{LS:flag}. 
They are defined recursively using divided difference operators.
The symmetric group $S_n$ acts on the polynomial ring $\mathbb{Z}[x_1,...,x_n]$ by permuting the variables: $w(f(x_1,...,x_n))=f(x_{w(1)},...,x_{w(n)})$.
With $s_i\in S_n$ being the simple reflection which maps $i$ to $i+1$, $i+1$ to $i$, and fixes all other elements, define the operators $\partial_i,\pi_i$ for $1\leq i\leq n-1$ by \begin{equation}\partial_i(f)=\frac{f-s_i(f)}{x_i-x_{i+1}}, \pi_i(f)=\partial_i((1-x_{i+1})f).\end{equation}
Then for $w_0\in S_n$ being the longest permutation, given by $w_0(i)=n+1-i$ for $1\leq i\leq n$, the Grothendieck polynomial for $w_0$ is $\mathcal{G}_{w_0}(x)=x_1^{n-1}x_2^{n-2}...x_{n-1}^1$.
Then the Grothendieck polynomials for other $w\in S_n$ are defined recursively by $\pi_i\mathcal{G}_{w}(x)=\mathcal{G}_{ws_i}(x)$ for $ws_i<w$ in the Bruhat order.
The equivariant version of these, double Grothendieck polynomials, use two sets of variables, $x_1,...,x_n$ and $T_1,...,T_n$, and start with $\mathcal{G}^T_{w_0}(x,T)=\prod_{i+j\leq n}(x_i+T_j-x_iT_j)$, then are recursively defined with the same operators as before, with the permutations acting trivially on the $T$ variables \cite{LS:Grothendieck,LS:flag}. 
Note that $\mathcal{G}^T_{w}(x,0)=\mathcal{G}_w(x)$ for all $w\in S_n$.

When $w\in S_n$ is $k$-Grassmannian, meaning $w(i)>w(i+1)$ for all $i\neq k$, it corresponds to the partition \[\lambda=(w(k)-k,w(k-1)-k+1,...,w(1)-1),\] and $\mathcal{G}_\lambda^T(x,T)$ is given by \cite{IN:factGrothendieck}: \[\mathcal{G}_\lambda^T(x,T)=\frac{\text{det}((x_i|T)^{\lambda_j+k-j}(1-x_i)^{j-1})}{\prod_{1\leq i<j\leq k}(x_i-x_j)},\] where $(x_i|T)^r=\prod_{j=1}^r(x_i+T_j-x_iT_j)$.
\subsection{Factorial Grothendieck polynomials represent Schubert classes}

To realize the polynomials geometrically, define $x_i$ such that $\lambda_{-1}(\mathcal{S})=\prod_{i=1}^kx_i$, and $T_i$ as $T_i=1-[\mathbb{C}^\vee_{t_i}]$.
In the Bott-Samelson variety, we define $z_i=1-[\mathcal{L}_i]$ for $1\leq i\leq k$.
Then, as a K-theory analog to Lemma \ref{pfpoly}, we have:
\begin{lem}\label{Kpfpoly}
For a polynomial $f(z,T)\in\mathbb{Z}[z_1,...,z_k;T_1,...,T_n]$ for $n=N+k$ with the above definitions of $z_i$ and $T_i$, \[\pi_*(f(z,T))=\sum_{w\in S_k}\left(\prod_{i=1}^{k-1}\prod_{j=i+1}^k\frac{(x_{w(i)}+T_{j-i}-x_{w(i)}T_{j-i})}{1-(1-x_{w(i)})(1-x_{w(j)})^{-1}}\right)f(x_{w(1)},...,x_{w(k)};T).\]
\end{lem}
\begin{proof}
Since $\mathcal{S}|_{e_\lambda}=\langle e_{i_1},...,e_{i_k}\rangle$ for $i_j=j+\lambda_{k+1-j}$ for $1\leq j\leq k$, we have that \[\prod_{j=1}^kx_j|_{e_\lambda}\lambda_{-1}(\mathcal{S})|_{e_\lambda}=\prod_{j=1}^k(1-e^{-t_{i_j}}).\]
As a result, we can say $x_j|_{e_\lambda}=1-e^{t_{i_j}}$.
Now since $T_j=1-[\mathbb{C}^\vee_{t_j}]$, we have $T_j|_{e_\lambda}=1-e^{-t_j}$.
With this, \[\begin{split}(x_{w(i)}+T_{j-i}- & x_{w(i)}T_{j-i})|_{e_\lambda}=(1-e^{t_{i_{w(i)}}})+(1-e^{-t_{j-i}})-(1-e^{t_{i_{w(i)}}})(1-e^{-t_{j-i}})\\ & =1+1-1-e^{t_{i_{w(i)}}}-e^{-t_{j-i}}+e^{t_{i_{w(i)}}}+e^{-t_{j-i}}-e^{t_{i_{w(i)}}}e^{-t_{j-i}}\\ & =1-e^{t_{i_{w(i)}}-t_{j-i}}.\end{split}\]
So then the localization of the right side of the equation at a fixed point $e_\lambda$ is given by \[\begin{split}& \sum_{w\in S_k}\left(\prod_{i=1}^{k-1}\prod_{j=i+1}^k\frac{(x_{w(i)}+T_{j-i}-x_{w(i)}T_{j-i})}{1-(1-x_{w(i)})(1-x_{w(j)})^{-1}}\right)f(x_{w(1)},...,x_{w(k)};T)|_{e_\lambda}\\ & =\sum_{w\in S_k}\left(\prod_{i=1}^{k-1}\prod_{j=i+1}^k\frac{(1-e^{t_{i_{w(i)}}-t_{j-i}})}{(1-e^{t_{i_{w(i)}}-t_{i_{w(j)}}})}\right)f(1-e^{t_{i_{w(1)}}},...,1-e^{t_{i_{w(k)}}};1-e^{-t})\end{split}\]
Using the identification $z_i=1-[\mathcal{L}_i]$, we have that $f(z,T)=f(1-[\mathcal{L}_1],...,1-[\mathcal{L}_k],1-e^{-t})$.
As a result, we can apply Lemma \ref{Kpfloc} to the left side of the equation and obtain \[\pi_*(f(z,T))|_{e_\lambda}=\sum_{w\in S_k}\left(\prod_{i=1}^{k-1}\prod_{j=i+1}^k\frac{(1-e^{t_{i_{w(i)}}-t_{j-i}})}{(1-e^{t_{i_{w(i)}}-t_{i_{w(j)}}})}\right)f(1-e^{t_{i_{w(1)}}},...,1-e^{t_{i_{w(k)}}};1-e^{-t}).\]
So then the localizations match at every fixed point and so by the injectivity of the localization map, the two classes are the same.
\end{proof}

\begin{remark}
This pushforward can also be expressed in terms of Demazure operators. See Proposition \ref{pushforwardoperator}.
\end{remark}

Since the factorial Grothendieck polynomials represent the classes of the structure sheaves of the corresponding Schubert varieties, we can recover the factorial Grothendieck polynomials as the pushforward of $\mathbb{V}^\emptyset$.
\begin{thm}\label{factGrothendieck}
For a partition $\lambda=(\lambda_1,...,\lambda_k)$ with corresponding k-Grassmannian permutation $w$ given by $w(i)=i+\lambda_{k+1-i}$ for $1\leq i\leq k$ and the remaining values arranged in ascending order,
\[[\mathcal{O}_{\Omega^\lambda}]=\pi_*([\mathcal{O}_{\mathbb{V}^\lambda}])=\mathcal{G}_w^T(x;T).\]
\end{thm}
\begin{proof}
Because $\Omega^\lambda$ has only rational singularities and $\pi:\mathbb{V}^\lambda\to\Omega^\lambda$ is a desingularization, we have that $\pi_*([\mathcal{O}_{\mathbb{V}^\lambda}])=[\mathcal{O}_{\Omega^\lambda}]$.
Using Lemma \ref{Kpf}, since $1-[\mathcal{L}_i\otimes\mathbb{C}^\vee_{t_j}]=1-(1-x_i)(1-T_j)=x_i+T_j-x_iT_j$, we have \[\pi_*([\mathcal{O}_{\mathbb{V}^\lambda}])=\pi_*(\prod_{i=1}^k\prod_{k+1-i}^{k-i+\lambda_i}(x_i+T_j-x_iT_j)).\]
Then applying Lemma \ref{Kpfpoly}, we have \[\begin{split}& \pi_*([\mathcal{O}_{\mathbb{V}^\lambda}])\\ & =\sum_{w\in S_k}\left(\prod_{i=1}^{k-1}\prod_{j=i+1}^k\frac{(x_{w(i)}+T_{j-i}-x_{w(i)}T_{j-i})}{1-(1-x_{w(i)})(1-x_{w(j)})^{-1}}\right)\left(\prod_{i=1}^k\prod_{j=k+1-i}^{k-i+\lambda_i}(x_{w(i)}+T_j-x_{w(i)}T_j)\right)\end{split}.\]
Making the simplification $\frac{1}{1-(1-x_{w(i)})(1-x_{w(j)})^{-1}}=\frac{(1-x_{w(j)})}{(1-x_{w(j)})-(1-x_{w(i)})}=\frac{1-x_{w(j)}}{x_{w(i)}-x_{w(j)}}$ yields \[\begin{split}\pi_*([\mathcal{O}_{\mathbb{V}^\lambda}])& =\\ & \sum_{w\in S_k}\left(\prod_{i=1}^{k-1}\prod_{j=i+1}^k\frac{(1-x_{w(j)})(x_{w(i)}+T_{j-i}-x_{w(i)}T_{j-i})}{x_{w(i)}-x_{w(j)}}\right)\\ & \hspace{8mm}\left(\prod_{i=1}^k\prod_{j=k+1-i}^{k-i+\lambda_i}(x_{w(i)}+T_j-x_{w(i)}T_j)\right)\end{split}.\]
Reindexing the products to put the $(x_{w(i)}+T_j-x_{w(i)}T_j)$ terms together, we have \[\begin{split}& \pi_*([\mathcal{O}_{\mathbb{V}^\lambda}])\\ & =\sum_{w\in S_k}\left(\prod_{i=1}^{k-1}\prod_{j=i+1}^k\frac{1-x_{w(j)}}{x_{w(i)}-x_{w(j)}}\right)\left(\prod_{i=1}^k\prod_{j=1}^{k-i+\lambda_i}(x_{w(i)}+T_j-x_{w(i)}T_j)\right).\end{split}\]
From here we make several simplifications, starting with \[\prod_{i=1}^{k-1}\prod_{j=i+1}^k\frac{1}{x_{w(i)}-x_{w(j)}}=\text{sgn}(w)\prod_{1\leq i<j\leq k}\frac{1}{x_i-x_j}.\]
Further, \[\begin{split}\prod_{i=1}^{k-1}\prod_{j=i+1}^k(1-x_{w(j)}) & =\prod_{1\leq i<j\leq k}(1-x_{w(j)})\\ & =\prod_{j=2}^k(1-x_{w(j)})^{j-1}\\ & =\prod_{i=1}^k(1-x_{w(i)})^{i-1}.\end{split}\]
Lastly, \[\prod_{i=1}^k\prod_{j=1}^{k-i+\lambda_i}(x_{w(i)}+T_j-x_{w(i)}T_j)=\prod_{i=1}^k(x_{w(i)}|T)^{k-i+\lambda_i}.\]
Putting all of these together gives \[\pi_*([\mathcal{O}_{\mathbb{V}^\lambda}])=\left(\prod_{1\leq i<j\leq k}\frac{1}{x_i-x_j}\right)\sum_{w\in S_k}\text{sgn}(w)\prod_{i=1}^k(x_{w(i)}|T)^{k-i+\lambda_i}(1-x_{w(i)})^{i-1}.\]
Since $\text{det}(a_{i,j})=\sum_{w\in S_k}\text{sgn}(w)\prod_{i=1}^ka_{w(i),i}$, this becomes \[\pi_*([\mathcal{O}_{\mathbb{V}^\lambda}])=\frac{\text{det}((x_i|T)^{\lambda_j+k-j}(1-x_i)^{j-1})}{\prod_{1\leq i<j\leq k}(x_i-x_j)},\]
which matches the factorial Grothendieck polynomial.
\end{proof}

This allows us to deal with $P_\lambda(z,T)$ when $\lambda$ is a partition, but if $\lambda$ is a composition, there is a straightening rule to express the pushforward as a combination of pushforwards of $P_\mu(z,T)$ for partitions $\mu$

\begin{cor}\label{SRuleK}
For any composition $\lambda=(\lambda_1,...,\lambda_k)$, let \[P_\lambda(z,T)=\prod_{i=1}^k\prod_{j=k+1-i}^{k-i+\lambda_i}(z_i+T_j-z_iT_j).\]
Then if $\lambda$ is a partition with corresponding $k$-Grassmannian permutation $w$, $\pi_*(P_\lambda(z,T))=\mathcal{G}_w^T(x,T)$.
If $\lambda$ is not a partition, then \[\pi_*(P_\lambda(z,T))=\sum_{j=\lambda_i+1}^{\lambda_{i+1}}\frac{1-T_{j+k-i}}{1-T_{\lambda_{i+1}+k-i}}\pi_*(P_{\mu^{(j)}}(z,T))-\sum_{j=\lambda_i+1}^{\lambda_{i+1}-1}\frac{1-T_{j+k-i}}{1-T_{\lambda_{i+1}+k-i}}\pi_*(P_{\nu^{(j)}}(z,T))\] in $K_T(\Gr(k,n))$, where $\mu^{(j)}=(\lambda_1,...,\lambda_{i+1},j,...,\lambda_k)$ and $\nu^{(j)}=(\lambda_1,...,\lambda_{i+1}-1,j,...,\lambda_k)$ for any $1\leq i\leq k$.
\end{cor}
\begin{proof}
The claim about partitions is Theorem \ref{factGrothendieck}.
The other claim boils down to a direct computation with the determinental formulae.
For simplicity, we will look at the computation for partitions with two parts (since the statement only involves switching two parts, the rest of it stays constant).
It centers around the identity \[\begin{split}& \pi_*(P_{(a,b)}(z,T))\\ & =\pi_*(P_{(a+1,b)}(z,T))+\frac{1-T_{a+2}}{1-T_{b+1}}\pi_*(P_{(b,a+1)}(z,T))-\frac{1-T_{a+2}}{1-T_{b+1}}\pi_*(P_{(b-1,a+1)}(z,T)).\end{split}\]
The indices on the $T$ variables here depend on the power $(x|T)^{\lambda_i+k-i}$ in the determinant, and in this simplified case we have $k=2$ and $i=1$. Therefore when applying this to the general situation, it would need to be $T_{a+1+k-i}$ and $T_{b+k-i}$ instead of $T_{a+2}$ and $T_{b+1}$.
By successively applying this, the result is obtained.
The claim is equivalent to \[\begin{split}& (1-T_{b+1})\pi_*(P_{(a,b)}(z,T))+(1-T_{a+2})\pi_*(P_{(b-1,a+1)}(z,T))\\ & =(1-T_{b+1})\pi_*(P_{(a+1,b)})+(1-T_{a+2})\pi_*(P_{(b,a+1)}(z,T)).\end{split}\]
From Theorem \ref{factGrothendieck}, we have $\pi_*(P_{(\lambda_1,\lambda_2)}(z,T))=\frac{det((x_i|T)^{\lambda_j+2-j}(1-x_i)^{j-1})}{x_1-x_2}$, and so by multiplying through by the common denominator the identity is equivalent to \[\begin{split}& (1-T_{b+1})((x_1|T)^{a+1}(x_2|T)^b(1-x_2)-(x_1|T)^b(x_2|T)^{a+1}(1-x_1))\\ & +(1-T_{a+2})((x_1|T)^b(x_2|T)^{a+1}(1-x_2)-(x_1|T)^{a+1}(x_2|T)^b(1-x_1))\\ & =(1-T_{b+1})((x_1|T)^{a+2}(x_2|T)^b(1-x_2)-(x_1|T)^b(x_2|T)^{a+2}(1-x_1))\\ & +(1-T_{a+2})((x_1|T)^{b+1}(x_2|T)^{a+1}(1-x_2)-(x_1|T)^{a+1}(x_2|T)^{b+1}(1-x_1)).\end{split}\]
For notational simplicity, we will use $G_T(a,b)=(x_1|T)^{a+1}(x_2|T)^b(1-x_2)-(x_1|T)^b(x_2|T)^{a+1}(1-x_1)$, which makes the left side of the equation \[\begin{split}& (1-T_{b+1})G_T(a,b)+(1-T_{a+2})G_T(b-1,a+1)\\ & =G_T(a,b)-T_{b+1}G_T(a,b)+G_T(b-1,a+1)-T_{a+2}G_T(b-1,a+1).\end{split}\]
Expanding yields \[\begin{split}& (x_1|T)^{a+1}(x_2|T)^b-x_2(x_1|T)^{a+1}(x_2|T)^b\\ &-(x_1|T)^b(x_2|T)^{a+1}+x_1(x_1|T)^b(x_2|T)^{a+1}-T_{b+1}G_T(a,b)\\ & +(x_1|T)^b(x_2|T)^{a+1}-x_2(x_1|T)^b(x_2|T)^{a+1}\\ & -(x_1|T)^{a+1}(x_2|T)^b+x_1(x_1|T)^{a+1}(x_2|T)^b-T_{a+2}G_T(b-1,a+1).\end{split}\]
Now since $x_i=(x_i+T_j-x_iT_j)-T_j(1-x_i)$, after some cancellations this simplifies to \[\begin{split}& -(x_1|T)^{a+1}(x_2|T)^{b+1}+T_{b+1}(x_1|T)^{a+1}(x_2|T)^b(1-x_2)\\ & +(x_1|T)^{b+1}(x_2|T)^{a+1}-T_{b+1}(x_1|T)^b(x_2|T)^{a+1}(1-x_1)-T_{b+1}G_T(a,b)\\ & -(x_1|T)^b(x_2|T)^{a+2}+T_{a+2}(x_1|T)^b(x_2|T)^{a+1}(1-x_2)\\ & +(x_1|T)^{a+2}(x_2|T)^b-T_{a+2}(x_1|T)^{a+1}(x_2|T)^b(1-x_1)-T_{a+2}G_T(b-1,a+1)\end{split}\]
Then we have $T_{b+1}G_T(a,b)=T_{b+1}((x_1|T)^{a+1}(x_2|T)^b(1-x_2)-(x_1|T)^b(x_2|T)^{a+1}(1-x_1))$ and $T_{a+2}G_T(b-1,a+1)=T_{a+2}((x_1|T)^b(x_2|T)^{a+1}(1-x_2)-(x_1|T)^{a+1}(x_2|T)^b(1-x_1))$, so these terms cancel and the left side of the equation becomes \[-(x_1|T)^{a+1}(x_2|T)^{b+1}+(x_1|T)^{b+1}(x_2|T)^{a+1}-(x_1|T)^b(x_2|T)^{a+2}+(x_1|T)^{a+2}(x_2|T)^b.\]
A similar process on the other side gives that the right side of the claim is also equal to this.
So then the claim holds, and as a result the proof is complete.
\end{proof}

\appendix
\section{An alternate interpretation of the pushforward formula}

Lemmas \ref{pfpoly} and \ref{Kpfpoly} can be expressed in terms of the divided difference operators from subsection 7.1.
In particular, for a permutation $w\in S_k$, there exists a reduced decomposition $w=s_{i_1}...s_{i_\ell}$, where $\ell$ is the number of inversions in $w$.
Then define operators $\partial_w=\partial_{i_1}...\partial_{i_\ell}$ and $\pi_w=\pi_{i_1}...\pi_{i_\ell}$, with $\partial_i$ and $\pi_i$ as defined in equation (1) in subsection 7.1.
These do not depend on the decomposition because the operators satisfy the same braid relations that the simple reflections do.
To prove that the pushforward can be expressed in terms of these operators, we recall some Lemmas regarding the combinatorics of the operators.

\begin{lem}
Let $f,g\in\mathbb{Z}[x_1,...,x_k]$.
\begin{enumerate}

\item (Leibniz rule) For any $1\leq i\leq k-1$, $\partial_i(fg)=f\partial_i(g)+s_i(g)\partial_i(f)$.

\item (Symmetry for $\partial$) For any $1\leq i\leq k-1$, if $f$ is symmetric with respect to $x_i$ and $x_{i+1}$, meaning $s_i(f)=f$, then $\partial_i(f)=0$. As  result, if $f$ is symmetric, then $\partial_i(fg)=f\partial_i(g)$.

\item (Symmetry for $\pi$) For any $1\leq k-1$, $\pi_i(f)=f+(1-x_i)\partial_i(f)$. As a result, if $f$ is symmetric, then $\pi_i(fg)=f\pi_i(g)$.

\end{enumerate}
\end{lem}





There is an operator known as the Jacobi symmetrizer, e.g. \cite[page 41]{FP:degenloci}, defined as \[\partial(f)=\sum_{w\in S_k}w\left(\frac{f}{\prod_{1\leq i<j\leq k}(x_i-x_j)}\right).\]
This operator can be expressed in terms of the divided difference operators $\partial_i$ as well.
The following result is known and can be proved through the Demazure character formula, but we have provided a direct combinatorial proof.

\begin{lem}\label{operator}
For any $f\in\mathbb{Z}[x_1,...,x_k]$ with $w_0\in S_k$ being the longest permutation given by $w_0(i)=k+1-i$ for $1\leq i\leq k$, \[\partial_{w_0}(f)=\sum_{w\in S_k}w\left(\frac{f}{\prod_{1\leq i<j\leq k}(x_i-x_j)}\right)\ \text{and}\ \pi_{w_0}(f)=\sum_{w\in S_k}w\left(f\prod_{1\leq i<j\leq k}\frac{1-x_j}{x_i-x_j}\right).\]
\end{lem}
\begin{proof}
We give the proof for the claim on $\partial_{w_0}$. The proof for $\pi_{w_0}$ is very similar, using the same techniques.

Let $w_0^{(p)}\in S_k$ denote the permutation given by $w_0^{(p)}(i)=p+1-i$ for $1\leq i\leq p$ and $w_0^{(p)}(i)=i$ for $p+1\leq i\leq k$.
We proceed by induction on $p$ in $w_0^{(p)}$.
The base case is $p=2$, since $w_0^{(1)}$ is the identity permutation.
By embedding $S_i$ into $S_k$ by $S_i=\langle s_1,...,s_{i-1}\rangle$ for $2\leq i\leq k$, the base case requires that $\partial_1(f)=\sum_{w\in S_2}w\left(\frac{f}{x_1-x_2}\right)$.
Since $S_2=\{1,s_1\}$, this is equivalent to \[\partial_1(f)=\frac{f}{x_1-x_2}+s_1\left(\frac{f}{x_1-x_2}\right)=\frac{f}{x_1-x_2}+\frac{s_1(f)}{x_2-x_1}=\frac{f-s_1(f)}{x_1-x_2},\] which is the definition of $\partial_1(f)$, so the base case holds.
Later on, we will want to use this to say that $\partial_i(f)=\frac{f}{x_i-x_{i+1}}+s_i\left(\frac{f}{x_i-x_{i+1}}\right)$ as an alternate definition.

For the induction step, we use the fact that $w_0^{(p+1)}=s_1...s_pw_0^{(p)}$.
As a result, $\partial_{w_0^{(p+1)}}=\partial_1...\partial_p\partial_{w_0^{(p)}}$.
We assume \[\partial_{w_0^{(p)}}(f)=\sum_{w\in S_p}w\left(\frac{f}{\prod_{1\leq i<j\leq p}(x_i-x_j)}\right)\] as the induction hypothesis.
For brevity's sake, define $g=\partial_{w_0^{(p)}}(f)$ and note that $g$ is symmetric in the variables $x_1,...,x_p$, due to it being the sum over $S_p$.

To complete the step, we use descending induction on $\partial_q...\partial_p\partial_{w_0^{(p)}}g$ to reach $\partial_1...\partial_p\partial_{w_0^{(p)}}g$.
We start with the base case of \[\partial_p\partial_{w_0^{(p)}}(f)=\frac{g}{x_p-x_{p+1}}+s_p\left(\frac{g}{x_p-x_{p+1}}\right).\]
For the induction step we assume \[\partial_q...\partial_pg=\sum_{i=q}^{p+1}s_i...s_p\left(\frac{g}{\prod_{j=q}^p(x_j-x_{p+1})}\right),\] where for $i=p+1$ the permutation $s_i...s_p$ is the identity.
Then we apply $\delta_{q-1}$ to this.
To do so, note that $s_{q+1}...s_pg$ is symmetric with respect to $x_1,...,x_q$, since $s_{q+1}$ through $s_p$ only affect $x_{q+1}$ through $x_{p+1}$ and $g$ is symmetric with respect to $x_1,...,x_p$ in the first place.
Similarly $\prod_{j=q+1}^n\frac{1}{x_j-x_{p+1}}$ does not involve the variables $x_1,...,x_q$ and so is symmetric with respect to those variables, except when acted upon by $s_q$.
In particular, the only term in the product $\prod_{j=q}^p\frac{1}{x_j-x_{p+1}}$, in all of these terms in the sum except the $i=q$ term, that isn't symmetric with respect to $x_{q-1}$ and $x_q$ is the $\frac{1}{x_q-x_{p+1}}$ term.
Then \[\begin{split}\partial_{q-1} & \left(s_i...s_p\left(\frac{g}{\prod_{j=q}^p(x_j-x_{n+1})}\right)\right)\\ & =s_i...s_p\left(\frac{g}{\prod_{j=q+1}^p(x_j-x_{n+1})}\right)\partial_{q-1}s_i...s_p\left(\frac{1}{x_q-x_{n+1}}\right)\end{split}\] for $q+1\leq i\leq p+1$.
Now we calculate \[\begin{split}\partial_{q-1}s_i...s_p\left(\frac{1}{x_q-x_{n+1}}\right) & =\partial_{q-1}\left(\frac{1}{x_q-x_i}\right)\\ & =\frac{1}{x_{q-1}-x_q}\left(\frac{1}{x_q-x_i}-s_{q-1}\left(\frac{1}{x_q-x_i}\right)\right)\\ & =\frac{1}{x_{q-1}-x_q}\left(\frac{(x_{q-1}-x_i)-(x_q-x_i)}{(x_q-x_i)(x_{q-1}-x_i)}\right)\\ & =\frac{1}{(x_q-x_i)(x_{q-1}-x_i)}\\ & =s_i...s_p\left(\frac{1}{(x_q-x_{p+1})(x_{q-1}-x_{p+1})}\right).\end{split}\]
With this, we have \[\partial_{q-1}\left(s_i...s_p\left(\frac{g}{\prod_{j=q}^p(x_j-x_{p+1})}\right)\right)=s_i...s_p\left(\frac{g}{\prod_{j=q-1}^p(x_j-x_{p+1})}\right)\] for $q+1\leq i\leq p+1$.
The only term left to calculate, then, is \[\begin{split}\partial_{q-1}s_q...s_p\left(\frac{g}{\prod_{j=q}^p(x_j-x_{n+1})}\right) & =\frac{1}{x_{q-1}-x_q}s_q...s_p\left(\frac{g}{\prod_{j=q}^p(x_j-x_{p+1})}\right)\\ & +s_{q-1}\left(\frac{1}{x_{q-1}-x_q}s_q...s_p\left(\frac{g}{\prod_{j=q}^p(x_j-x_{p+1})}\right)\right).\end{split}\]
For this, we use the fact that $\frac{1}{x_{q-1}-x_q}=s_q...s_p\left(\frac{1}{x_{q-1}-x_{p+1}}\right)$ to get that \[\begin{split}\partial_{q-1}s_q...s_p\left(\frac{g}{\prod_{j=q}^p(x_j-x_{p+1})}\right) & =s_q...s_p\left(\frac{g}{\prod_{j=q-1}^p(x_j-x_{p+1})}\right)\\ & +s_{q-1}s_q...s_p\left(\frac{g}{\prod_{j=q-1}^p(x_j-x_{p+1})}\right).\end{split}\]
Combining these results yields \[\partial_{q-1}...\partial_pg=\sum_{i=q-1}^{p+1}s_i...s_p\left(\frac{g}{\prod_{j=q-1}^p(x_j-x_{p+1})}\right).\]
This completes the induction step, which proves \[\partial_{w_0^{(p+1)}}=\partial_1...\partial_pg=\sum_{i=1}^{p+1}s_i...s_p\left(\frac{g}{\prod_{j=1}^p(x_j-x_{p+1})}\right).\]
Then plugging in the induction hypothesis for $g$ gives \[\partial_{w_0^{(p+1)}}=\sum_{w\in S_p}\sum_{q=1}^{p+1}s_q...s_pw\left(\frac{f}{\prod_{1\leq i<j\leq p+1}(x_i-x_j)}\right).\]
Now since $s_q...s_pw$ ranges over $S_{p+1}$ as $w$ ranges over $S_p$ and $1\leq q\leq p+1$, this is equivalent to \[\partial_{w_0^{(p+1)}}=\sum_{w\in S_{p+1}}w\left(\frac{f}{\prod_{1\leq i<j\leq p+1}(x_i-x_j)}\right).\]
This completes the induction step, which completes the proof.
\end{proof}

Recalling that $\partial_i$ and $\pi_i$ extend linearly over $\mathbb{Z}[t_1,...,t_{N+k}]$ and applying this to Lemmas \ref{pfpoly} and \ref{Kpfpoly}, we have:

\begin{prop}\label{pushforwardoperator}
With the definitions from section 3, for $f\in\mathbb{Z}[z_1,...,z_k][t_1,...,t_{N+k}]$, we have \[\pi_*(f(z;t))=\partial_{w_0}(p_\delta f(x;t)),\] where $w_0\in S_k$ is the longest permutation and $p_\delta=\prod_{i=1}^{k-1}\prod_{j=1}^{k-i}(x_i+t_j)$.

With the definitions from section 7, for $f\in\mathbb{Z}[z_1,...,z_k][T_1,...,T_{N+k}]$, we have \[\pi_*(f(z;T))=\pi_{w_0}(P_\delta f(x;T)),\] where $P_\delta=\prod_{i=1}^{k-1}\prod_{j=1}^{k-i}(x_i+T_j-x_iT_j)$.
\end{prop}
\begin{proof}
Apply Lemma \ref{operator} to Lemma \ref{pfpoly} and Lemma \ref{Kpfpoly}.
\end{proof}

\bibliographystyle{halpha}
\bibliography{Bott_Samelson_Factorial_Schur_Grothendieck}

\end{document}